\documentclass[11pt]{amsart}
\usepackage{amsmath,amssymb,amsthm}
\usepackage{array,booktabs,graphicx}
\usepackage{xcolor}
\usepackage[colorlinks=true,linkcolor=blue,citecolor=blue,urlcolor=blue]{hyperref}
\usepackage{comment}
\usepackage{tikz-cd}

\theoremstyle{plain}
\newtheorem{theorem}{Theorem}[section]
\newtheorem{lemma}[theorem]{Lemma}
\newtheorem{proposition}[theorem]{Proposition}
\newtheorem{corollary}[theorem]{Corollary}
\newtheorem{conjecture}[theorem]{Conjecture}

\theoremstyle{definition}
\newtheorem{definition}[theorem]{Definition}
\newtheorem{example}[theorem]{Example}

\theoremstyle{remark}
\newtheorem{remark}[theorem]{Remark}

\title[New points on covers of elliptic curves]{New algebraic points on
covers of elliptic curves}

\author{Diana Mocanu}
\address{Max Planck Institute for Mathematics, Bonn,
Germany}
\email{d.mocanu@mpim-bonn.mpg.de}

\author{George C. \c Turca\c s}
\address{Babe\c s-Bolyai
University, Cluj-Napoca, Romania}
\email{george.turcas@ubbcluj.ro}

\date{\today}

\newcommand{\Q}{\mathbb{Q}}
\newcommand{\Qbar}{\overline{\mathbb{Q}}}
\newcommand{\GL}{\mathrm{GL}}
\newcommand{\Z}{\mathbb{Z}}
\newcommand{\PP}{\mathbb{P}}
\newcommand{\FF}{\mathbb{F}}
\newcommand{\Gal}{\operatorname{Gal}}
\newcommand{\dd}{\mathfrak{d}}
\newcommand{\DD}{\mathfrak{D}}
\newcommand{\news}{\mathrm{new}}
\newcommand{\OO}{\mathcal{O}}
\newcommand{\Spec}{\operatorname{Spec}}
\newcommand{\supp}{\operatorname{supp}}
\newcommand{\Disc}{\operatorname{Disc}}
\newcommand{\ord}{\operatorname{ord}}
\newcommand{\Frob}{\operatorname{Frob}}

\DeclareMathOperator{\Tr}{\text{Tr}}

\DeclareMathOperator{\End}{\text{End}}

\begin{document}
\begin{abstract}
    For a smooth projective curve $C/\Q$ of genus $\geq 2$ and $L/\Q$ an  extension, we write
$C(L)_{\text{new}}=\{P\in C(L):\Q(P)=L\}$. Recent work of
Khawaja and Siksek conjectures that this set is empty for $100\%$ of
degree $n$ number fields $L$, when ordered by absolute discriminant. Moreover, they bring evidence towards this conjecture when $C$ is a degree $n$ cover of $\PP^1$. We complement their work by proving analogous results for degree $n$ covers $\psi:C\to E$ of elliptic curves $E$. 

Our main result shows that, under suitable hypotheses, the number of distinct absolute discriminants at most $X$ of primitive degree $n$ fields $L$ with $C(L)_{\text{new}}\neq\varnothing$ is $O(X^{1/2})$ or $O(X/(\log X)^{\alpha})$, for some $\alpha>0$. In degrees $2,3,4$ and $5$ we show that these fields have density $0$ among all fields of the same degree (in degree $4$, also among the primitive ones). The novelty is for degrees $4$ and $5$, where we use work of Bhargava--Shankar--Wang and McGown--Thorne--Tucker to count fields with specified local constraints.

Moreover, we give concrete examples of $8$ bielliptic modular curves $X_0(N)$, for which $X_0(N)(L)_{\text{new}}=\emptyset$ 
for $100 \%$ of quadratic fields $L$. Lastly, we point out modular covers of degrees $3$ and $5$ in the LMFDB for which similar conclusions hold.
\end{abstract}
\maketitle
\section{Introduction}
Let $C/\Q$ be a nice (i.e. smooth, projective, absolutely irreducible) curve of genus at
least $2$. By Faltings' Theorem \cite{Fal}, $C$ has finitely many points over any
fixed number field, but as the field varies infinitely many points can
appear, and one may ask over which fields they first appear. Fix an
algebraic closure $\Qbar$ of $\Q$.
We write $\Q(P)$ for the smallest
field over which $P$ is defined. We say $P$ is \emph{primitive} whenever $\Q(P)/\Q$ has no intermediate subextension.
For a number field $L/\Q$, write
$$C(L)_{\news} = \{ P \in C(L): \Q(P) = L \}$$
for the set of \emph{new} points over $L$: those defined over $L$ and
over no proper subfield. We say that $L$ \emph{supports a new point} if
$C(L)_{\news}\neq\varnothing$. For a curve $C/K$ over a number field $K$ and an extension $L/K$, we use the same notation relative to $K$: $C(L)_{\news}=\{P\in C(L): K(P)=L\}$, and a closed point is primitive if $K(P)/K$ has no proper intermediate field.  In~\cite{KS}, Khawaja and Siksek pose the following conjecture, which we state for the family of degree $n$ field extensions.

\begin{conjecture}[Khawaja--Siksek {\cite[Conjecture~3]{KS}}]\label{conjecture}
Let \(K\) be a number field, and \(C/K\) be a nice curve of genus at least \(2\). Fix an integer $n\geq 2$. Then
\(
C(L)_{\news}=\varnothing
\)
for \(100\%\) of \(L/K\) of degree $n$, ordered by their absolute
discriminants.
\end{conjecture}
 Note that a geometric
source may produce infinitely many degree $n$ points, and even infinitely
many fields of definition, while the set of those fields has density $0$
among all degree $n$ fields. The conjecture is a statistical analogue of the
Diophantine stability studied by Mazur and Rubin \cite{MR}.

 Work of
Harris and Silverman~\cite{HS} proves that a smooth
projective curve $C/\mathbb Q$ of genus at least $2$ with infinitely
many quadratic points is either hyperelliptic or bielliptic. 
More recently, Khawaja and Siksek~\cite{KSsingle} prove that when 
$g(C)>(n-1)^2$ for $n\geq3$ (or $g(C)\geq3$
for $n=2$) and the Jacobian of $C$
is simple or has finite Mordell--Weil group, an infinite source of primitive degree $n$ points comes from covers of $\PP^1$.
A generalisation due to Derickx~\cite[Corollary 1.4 (1)]{Deric} gives the following.
\begin{theorem}\label{thm:Deric}
    Let $K$ be a number field, \(C/K\) a smooth projective geometrically integral curve, and
let \(n\geq 2\). Suppose that \(C\) has infinitely many primitive
closed points of degree \(n\) and that
\[
g(C)>(2n-1)^2.
\]
Then there is a smooth projective geometrically integral curve \(Y/K\)
with \(Y(K)\) infinite, and an indecomposable morphism
\(\psi:C\to Y\) of degree \(n\), such that all but finitely many
primitive degree \(n\) points of \(C\) are of the form
\(\psi^*(y)\) for some \(y\in Y(K)\).
\end{theorem}

Thus, if the genus is large enough, the only possible infinite primitive sources of points on a nice curve $C/\Q$
are degree $n$ maps to $\PP^1$ or to a positive rank elliptic curve. For similar results on lower genus in degrees
$4$ and $5$, when the degree is the least one occurring
infinitely often, see Kadets--Vogt
\cite[Theorems~1.2 and~1.3]{KadetsVogt}. 

In recent work, 
Khawaja and Siksek~\cite[Theorem~9]{KS}
give evidence towards Conjecture~\ref{conjecture} in the case $\psi:C\to \PP^1$, assuming certain conditions on the squarefree part of the
discriminant of the polynomial defining this cover. 
More precisely, they prove that the norms of discriminants of the primitive degree $n$ fields supporting new points form a set of density $0$ in the positive integers.

We study the analogous elliptic curve setting, which is the first proposed project in \textit{loc. cit.}
We consider a finite morphism
\[
\psi\colon C\longrightarrow E
\]
of degree $n$, where $E/\mathbb Q$ is an elliptic curve. When
$E(\mathbb Q)$ has positive rank, the rational fibres of $\psi$ may
give rise to infinitely many degree $n$ fields. 
 
The projective line argument in \textit{loc. cit.} relies on using an affine coordinate $T$, which turns rational fibres into one-variable polynomial specialisations at $T=t$. This allows their discriminants to be tracked through discriminants of the one variable polynomials in question. For an elliptic curve there is no analogous preferred global coordinate, so this description cannot be replicated. We therefore replace the polynomial discriminant by the discriminant divisor
of the finite cover $\psi:C \to E$ denoted by $D_\psi$. This is defined from the trace pairing on $\psi_*\mathcal O_C$. Its odd part, $B_{\psi}$, plays the role of the squarefree factor appearing in the $\PP^1$ setting; see Section~\ref{sec:language} for the precise definitions. 

\subsection{Our results}
Before we present our main result, we introduce some notation. 
A number field $L/\Q$ is \emph{primitive} if it has no proper
intermediate field. A closed point $P$ is \emph{primitive}
if $\Q(P)/\Q$ is primitive.

A permutation of a finite set is a
\emph{derangement} if it has no fixed point. We say that a group action
\emph{contains a derangement} if some element of its image is a
derangement. 

\begin{theorem}\label{thm:main}
Let $n\geq2$, and let $\psi:C\to E$ be a finite morphism of degree
$n$ from a smooth projective absolutely irreducible curve $C/\Q$ of
genus $g$ to an elliptic curve $E/\Q$. Suppose that
$g>n^2+1$
and either $B_\psi=\varnothing$, or the action of
$\Gal(\Qbar/\Q)$ on $B_\psi(\Qbar)$ contains a derangement. For $X\geq2$, let
\[
\mathcal D_{C,n}(X)
=
\left\{
|\Delta_L|\leq X:
[L:\Q]=n,\ L/\Q\text{ primitive}, 
C(L)_{\news}\neq\varnothing\right\}.
\]
Then the following hold.
\begin{enumerate}
\item If $B_\psi=\varnothing$, then
$\#\mathcal D_{C,n}(X)\ll_\psi X^{1/2}$.
\item If $B_\psi\neq\varnothing$, let $L_\psi/\Q$ be the finite
Galois extension cut out by the action of $\Gal(\Qbar/\Q)$ on
$B_\psi(\Qbar)$. Write
$
G_\psi=\Gal(L_\psi/\Q)$
and
$$
 \alpha_\psi
 :=
 \frac{
 \#\{\sigma\in G_\psi:B_\psi(\Qbar)^\sigma=\varnothing\}
 }{\#G_\psi}.$$
Then $0<\alpha_\psi<1$ and
 $\#\mathcal D_{C,n}(X)
 \ll_\psi
 \frac{X}{(\log X)^{\alpha_\psi}}$.
\end{enumerate} 
\end{theorem}

The idea of the proof is as follows. Firstly, we show that, apart from finitely many exceptions, primitive fields $L$ supporting new points must satisfy $v_p(\Delta_L)\equiv0\pmod 2$ for all $p$ in a set $\mathcal S$ of primes of regular density $\alpha>0$ (Theorem~\ref{thm:parity}). This uses Song--Tucker's finiteness theorem \cite[Proposition~2.3]{ST}, a spreading out argument for the cover $\psi$ and, when $B_\psi\neq\varnothing$, and a version of Chebotar\"ev's Density Theorem (Lemma~\ref{lem:avoid}). Then we count the integers $m\leq X$ with $v_p(m)\equiv0\pmod 2$ for all $p\in\mathcal S$: directly when $B_\psi=\varnothing$, and using \cite[Theorem~10]{KS} otherwise.

\begin{theorem}
\label{thm:field density}
Under the hypotheses and notation of Theorem~\ref{thm:main}, the
following hold.
\begin{enumerate}
\item If $n\in\{2,3,5\}$, the degree $n$ fields $L/\Q$ for
which
$C(L)_{\news}\neq\varnothing$
have density $0$ among all degree $n$ fields ordered by absolute
discriminant.

\item If $n=4$, the primitive quartic fields $L/\Q$ for which
$C(L)_{\news}\neq~\varnothing$
have density $0$ both among all quartic fields and among the primitive
quartic fields, in each case ordered by absolute discriminant.
\end{enumerate}
\end{theorem}

The degree $4$ and $5$ cases rest on two statements about number
fields which may be of independent
interest (Theorems~\ref{thm:qqext-quintic-parity}
and~\ref{thm:qqext-quartic-parity}). Suppose $\mathcal S$ is a regular set of
primes of positive density. Then the density of quintic fields, respectively the
primitive quartic fields, whose discriminant has even valuation at every
$p\in\mathcal S$ is $0$. 
For counting quintic fields we use work of McGown--Thorne--Tucker
\cite[Theorem~6]{MTT23}, and for primitive quartic fields we use Bhargava--Shankar--Wang \cite[Theorem~2]{BSWglobal}, together with the bound for $A_4$-quartic fields in \cite[Theorem~1.4]{BSTTTZ}.

This extends the cubic result
of \cite[Corollary~18]{KS}, answering the second project in
\textit{loc. cit.} We also obtain the
following complement to Theorem~9 in \textit{loc. cit.}
\begin{corollary}\label{cor:P1}
Let $n\in\{4,5\}$, and let $\psi:C\to\PP^1$ be a morphism of degree
$n$ satisfying the hypotheses of \cite[Theorem~9]{KS}: in the notation of
Proposition~\ref{prop:KS-dictionary}, $g(C)>n^2-n+1$, $\deg g$ is
even, and $\Gal(g)$ contains a derangement (i.e. an element acting freely on the roots of $g$). 
\begin{itemize}
    \item If $n=5$, the quintic
fields $L$ with $C(L)_{\news}\neq\varnothing$ have density $0$ among
all quintic fields.
\item If $n=4$, the primitive quartic fields $L$ with
$C(L)_{\news}\neq\varnothing$ have density $0$ both among all quartic
fields and among the primitive quartic fields.
\end{itemize}

\end{corollary}

In Section~\ref{sec:modular-quadratic}, we apply our main theorem for the class of bielliptic modular curves and we get the following. 

\begin{corollary}\label{cor:X0}
Let
$
N\in\{53,61,65,79,83,89,101,131\}.
$
Then
\[
X_0(N)(L)_{\news}=\varnothing
\]
for $100\%$ of quadratic fields $L/\Q$, when ordered by absolute
discriminant.  More precisely, there exists
$\alpha_N\in(0,1)$ such that
\[
\#\left\{
L/\Q\text{ quadratic}:
|\Delta_L|\leq X,\ 
X_0(N)(L)_{\news}\neq\varnothing
\right\}
\ll_N
\frac{X}{(\log X)^{\alpha_N}}.
\]
\end{corollary}
For double covers, $B_\psi$ is the branch divisor (Example~\ref{ex:double-cover}), and all the modular curves above are bielliptic via the
Atkin--Lehner quotient $X_0(N)\to X_0^+(N)$ (i.e. $X_0^+(N)$ is an elliptic curve). It is well known that its branch points correspond to CM elliptic curves, and then classical CM theory supplies
the required Galois derangement. The genus hypothesis holds except
at levels $53$, $61$ and $65$, where Box's results
\cite[\S\S~4.2, 4.3 and 4.5]{Box} and
Remark~\ref{rem:Xi-finite} give the conclusion.

Earlier work \cite{Box,NajmanVukorepa} classifies these quadratic
points geometrically; Corollary~\ref{cor:X0} complements this by
counting their fields of definition by discriminant.

Modular curves also provide covers of degrees $3$ and $5$ to which
Theorem~\ref{thm:main} applies; two examples taken from the LMFDB
\cite{LMFDB} are described in Remark~\ref{rem:higher-degree modular}.
Together with Corollary~\ref{cor:X0}, this contributes to the third
project of \cite[Section~8]{KS}.

\subsection{Organisation of the paper} In
Section~\ref{sec:language}, we define the discriminant divisor of the cover and its odd
locus. Moreover, we use Chebotar{\"e}v's Theorem to find primes at which the odd locus has no rational point. In Section~\ref{sec:parity} we introduce the algebraic geometry machinery needed to prove Theorem~\ref{thm:main}, which concludes the section.
Section~\ref{sec:qqext} establishes the field counting
results in degrees $4$ and $5$ and proves
Theorem~\ref{thm:field density} and Corollary~\ref{cor:P1}.
Section~\ref{sec:modular-quadratic} proves the application to bielliptic modular
curves using CM theory.

\subsection*{Acknowledgements}
The first author completed this work at the Max Planck Institute for Mathematics in Bonn and appreciates its hospitality and financial support.
We thank Mugurel Barc\u{a}u for his careful reading of
Sections~\ref{sec:language} and~\ref{sec:parity} and for many helpful
suggestions that clarified the arguments and improved the exposition. We are also grateful to Vlad Matei and Alex Pascadi for their valuable
feedback on Section~\ref{sec:qqext} and to Maleeha Khawaja for bringing Derickx’s single-source theorem (Theorem~\ref{thm:Deric}) to our attention.

\subsection*{Use of artificial intelligence tools}
The authors used OpenAI's ChatGPT and Anthropic's Claude (several models over time, including GPT 6 Astra and Opus 5.5) in long-term conversations to help locate and organise the relevant literature, to critically review portions of the mathematical arguments and to improve exposition. OpenAI Codex was also used for auxiliary computational work related to Remark~\ref{rem:higher-degree modular}. The authors carefully checked every source, computation and proof step found or produced with the help of these systems, and take full responsibility for the contents of the article.

\section{Discriminants of finite covers}
\label{sec:language}

This section is dedicated to the
discriminant divisor of a cover $\psi: C\to E$, which should be viewed as the coordinate-free version of the discriminant of the defining polynomial of the cover.

Let $k$ be a field of characteristic zero.
For a closed point $b$ of a curve $E/k$, we write
$k(b):=\mathcal O_{E,b}/\mathfrak m_b$ for its residue field,
a finite extension of $k$, and let $\deg(b)=[k(b):k]$.
A choice of geometric point $\widetilde b\in E(\bar k)$
above $b$ determines a $k$-embedding $k(b)\hookrightarrow\bar k$,
whose image is the field of definition $k(\widetilde b)$.
When $k=\mathbb Q$, we use the notation $\mathbb Q(b)$.

\subsection{The discriminant divisor}

 Let $\psi:C\to E$ be a
finite morphism of degree $n$ between smooth projective geometrically
integral curves over $k$. 

Let $\mathcal A:=\psi_*\mathcal O_C$.
Since $\psi$ is finite, $\mathcal A$ is a finite quasi-coherent
$\mathcal O_E$-algebra. Since $C$ is integral and $\psi$ is
dominant, $\mathcal A$ is torsion free as an $\mathcal O_E$-module.
For every point $b\in E$, the local ring $\mathcal O_{E,b}$ is
either a field or a discrete valuation ring, because $E$ is a regular
curve. Hence $\mathcal A_b$ is a free $\mathcal O_{E,b}$-module of finite rank. It follows that $\mathcal A$ is locally
free. Its rank is $n=[k(C):k(E)]$, so $\psi$ is finite flat of degree $n$.

\begin{definition}\label{def:disc}
On an open set on which $\mathcal A$ has basis $e_1,\ldots,e_n$, form
the determinant
\[
 \det\bigl(\operatorname{Tr}_{\mathcal A/\OO_E}(e_i e_j)\bigr)_{1\leq i,j\leq n}.
\]
A change of basis multiplies this determinant by the square of a unit,
so these local generators define a locally principal ideal sheaf
$\dd_\psi\subset\OO_E$, called the \emph{discriminant ideal} of
$\psi$. Let $\eta$ be the generic point of $E$. The trace pairing
on $\mathcal A_\eta=k(C)$ is the field trace pairing for
$k(C)/k(E)$. Since this extension is separable, the pairing is
nondegenerate, so its determinant is nonzero. Thus
$(\dd_\psi)_\eta\neq0$. As $E$ is integral, the local generators of
$\dd_\psi$ are non-zero divisors. Hence $\dd_\psi$ is an invertible
ideal sheaf and determines an effective divisor
\begin{equation} \label{def:disc-divisor}
 D_\psi
 =
 \sum_{b\in E^{(1)}}m_b[b],
 \qquad
 m_b=\ord_b(\dd_\psi),
\end{equation}
where $E^{(1)}$ denotes the set of closed points of $E$. We call
$D_\psi$ the \emph{discriminant divisor} of $\psi$.
\end{definition}

After base change to a fixed algebraic closure $\bar k$ of $k$, we
write
\begin{equation}\label{eq:disc-base-change}
 (D_\psi)_{\bar k}
 =
 \sum_{b\in E(\bar k)}m_b[b].
\end{equation}
The trace pairing and its determinant commute with extension of the base
field, so $(D_\psi)_{\bar k}=D_{\psi_{\bar k}}$. The degree of a
divisor is invariant under extension of the base field; in particular,
\[
 \deg\bigl((D_\psi)_{\bar k}\bigr)=\deg D_\psi.
\]
For $P\in C(\bar k)$, let $e_P$ denote the ramification index of
$\psi$ at $P$.

\begin{lemma}\label{lem:disc-ramification}
For every $b\in E(\bar k)$, one has
\[
 m_b
 =
 \sum_{\substack{P\in C(\bar k)\\\psi(P)=b}}(e_P-1).
\]
Consequently,
\[
 \deg D_\psi
 =
 2g(C)-2-n\bigl(2g(E)-2\bigr).
\]
\end{lemma}

\begin{proof}
Let $\DD_\psi$ denote the different divisor of $\psi$; see
\cite[\S7.4.2, pp.~289--291]{Liu}. Since $\operatorname{char}(k)=0$,
the cover is tamely ramified and the coefficient of $\DD_\psi$ at a
point $P$ is $e_P-1$. The local discriminant--different formula,
applied over each closed point $b\in E$ to the finite extension of
$\mathcal O_{E,b}$ to the semilocal Dedekind domain
$(\psi_*\mathcal O_C)_b$, gives
\[
 D_\psi=\psi_*\DD_\psi;
\]
see \cite[Chapter~III, Section~2]{Neu} and
\cite[Tags~0BVH and~0C17]{Stacks}. Applying the same identity after base
change to $\bar k$, where all residue degrees are $1$, yields the
formula for $m_b$. Taking degrees and applying the Riemann--Hurwitz
formula gives the final assertion.
\end{proof}

In particular, when $E$ is an elliptic curve,
\begin{equation}\label{eq:deg-discr}
 \deg D_\psi=2g(C)-2.
\end{equation}

\begin{definition}\label{def:odd}
The \emph{odd locus} of the discriminant divisor is the reduced closed
subscheme
\[
 B_\psi
 :=
 \coprod_{\substack{b\in E^{(1)}\\m_b\text{ odd}}}
 \Spec k(b)
 \hookrightarrow E.
\]
Since $E$ is regular of dimension $1$, the reduced closed subscheme
$B_\psi$ determines the reduced effective divisor
\[
 \sum_{\substack{b\in E^{(1)}\\m_b\text{ odd}}}[b].
\]
By abuse of notation, we also write $B_\psi$ for this divisor.
We refer to either of them as the \emph{odd part} of $D_\psi$.
\end{definition}

After base change to $\bar k$, its geometric points form the finite set
\[
 B_\psi(\bar k)
 =
 \{b\in E(\bar k):m_b\text{ is odd}\}.
\]
This set carries its natural $G_k$-action, where
$G_k:=\Gal(\bar k/k)$. Moreover,
$
 \deg B_\psi\equiv\deg D_\psi\pmod2$.
Thus, when $E$ is an elliptic curve, \eqref{eq:deg-discr} gives
$\deg B_\psi\equiv0\pmod2$.

\begin{remark}\label{rem:connected-odd-locus}
Suppose that $E$ is an elliptic curve and that $B_\psi$ is nonempty
and connected over $k$. Since $B_\psi$ is finite and reduced, this
means that it consists of a single closed point of $E$, not necessarily
a $k$-rational point. Equivalently, the action of $G_k$ on
$B_\psi(\bar k)$ is transitive. Moreover, $\deg B_\psi$ is even and
hence at least $2$. Jordan's Theorem therefore shows that this action
contains a derangement (see \cite{Ser03}). In
particular, the Galois-theoretic hypothesis in Theorem~\ref{thm:main} is
automatic in this case.
\end{remark}

\subsection{The projective line case}

We now show that Definition~\ref{def:odd} captures, without coordinates,
both hypotheses imposed on the squarefree polynomial in
\cite[Theorem~9]{KS}.

\begin{proposition}\label{prop:KS-dictionary}
In the notation above, take $k=\Q$ and $E=\PP^1$, so that $\psi:C\to\PP^1$. Choose an affine coordinate $T$ on
$\PP^1$, with corresponding point at infinity $\infty$. Let
$Y\in\Q(C)$ satisfy $\Q(C)=\Q(T)(Y)$, and choose
$F(T,Y)\in\Q[T,Y]$, irreducible of degree $n$ in $Y$,
such that $F(T,Y)=0$ in $\Q(C)$. Write
\[
 \Delta_Y(F)=g(T)h(T)^2,
\]
where $g,h\in\Q[T]$ and $g$ is squarefree. Write $\Gal(g)$ for
the Galois group of the splitting field of $g$, acting on its roots. Then:
\begin{enumerate}
\item the finite geometric points of $B_\psi$ are exactly the roots
of $g$;
\item $\infty\in B_\psi(\Q)$ if and only if $\deg g$ is odd;
\item the conditions that $\deg g$ be even and that the action of
$\Gal(g)$ on the roots of $g$ contain a derangement are equivalent
to the condition that the action of $G_\Q$ on
$B_\psi(\Qbar)$ contain a derangement.
\end{enumerate}
\end{proposition}

\begin{proof}
Let $a_n(T)$ be the leading coefficient of $F$ as a polynomial in
$Y$. Over $K=\Q(T)$, the monic minimal polynomial of $Y$ is
$a_n^{-1}F$. Consequently, the discriminant of the $K$-basis
$\boldsymbol\beta=(1,Y,\ldots,Y^{n-1})$ is
\[
 \operatorname{disc}(\boldsymbol\beta)
 =
 a_n(T)^{-2n+2}\Delta_Y(F).
\]

Fix a place $b$ of $\PP^1$, possibly $b=\infty$, and let
$A_b=\OO_{\PP^1,b}$. The finite flat $A_b$-algebra
$(\psi_*\OO_C)_b$ has an $A_b$-basis
$\mathbf e=(e_1,\ldots,e_n)$. After tensoring with $K$, this is a
$K$-basis of $\Q(C)$. Thus, although $\boldsymbol\beta$ need not be
an $A_b$-basis, there is some $M\in\GL_n(K)$ such that
$\boldsymbol\beta=\mathbf eM$. Hence
\[
 \operatorname{disc}(\boldsymbol\beta)
 =
 \det(M)^2\operatorname{disc}(\mathbf e).
\]
The valuation of $\operatorname{disc}(\mathbf e)$ is the coefficient
$m_b$ of the discriminant divisor. Since both the exponent $2n-2$ of
$a_n$ and the exponent of $\det(M)$ are even, we obtain
\[
 m_b\equiv\ord_b\Delta_Y(F)\pmod2.
\]

As $\Delta_Y(F)=gh^2$ and $g$ is squarefree, its valuation at a
finite geometric point $\alpha$ is odd precisely when $g(\alpha)=0$,
proving (1). At infinity,
\[
 \ord_\infty\Delta_Y(F)
 \equiv
 \ord_\infty g
 \equiv
 -\deg g
 \pmod2,
\]
which proves (2).

If $\infty\in B_\psi(\Q)$, then it is fixed by every Galois element.
If $\infty\notin B_\psi(\Q)$, then the Galois set
$B_\psi(\Qbar)$ is exactly the root set of $g$. This proves (3).
\end{proof}

The dictionary between the projective line setting of \cite{KS} and the
elliptic curve setting of this work is therefore as follows. The cover
$\psi:C\to\PP^1$, together with an affine coordinate $T$, is
replaced by the coordinate-free cover $\psi:C\to E$. Moreover, the polynomial
discriminant $\Delta_Y(F)$ is replaced by the discriminant divisor
$D_\psi$, and its squarefree part $g$ is replaced by the odd part
$B_\psi$. Accordingly, the condition that $\deg g$ be even and that the action
on the roots of $g$ contains a derangement is replaced by the condition that the action of $G_\Q$ on
$B_\psi(\Qbar)$ contain a derangement.

Finally, the specialisation $T=a/b$ used in \cite{KS} is replaced below
by pullback along a section of an integral model of $E$; see
Lemma~\ref{lem:bc} and the proof of Theorem~\ref{thm:parity}.

\subsection{Double covers}

The odd locus has a particularly concrete description for a double cover.

\begin{example}\label{ex:double-cover}
Assume that $\deg\psi=2$. There
is a class
$\delta\in k(E)^\times/k(E)^{\times2}$
such that
$k(C)=k(E)(\sqrt\delta)$.
Choose a representative and write
\[
 \operatorname{div}(\delta)=\sum_Q m_Q[Q].
\]
At a closed point $Q$, choose a uniformiser $t$ and write
$\delta=ut^{m_Q}$, with $u$ a unit. The quadratic extension is
ramified at $Q$ if and only if $m_Q$ is odd. Since the ramification is
tame, Lemma~\ref{lem:disc-ramification} gives
$$
 D_\psi
 =
 B_\psi
 =
 \sum\limits_{m_Q\text{ odd}}[Q].$$
Thus, for a double cover, the odd part of the discriminant divisor is
precisely the branch divisor. This description is unchanged when
$\delta$ is multiplied by a square; in particular, odd poles as well as
odd zeros must be included.

For $R\in E(k)\setminus\supp(\operatorname{div}(\delta))$, the fibre is
\[
 C_R\simeq\Spec\bigl(k[w]/(w^2-\delta(R))\bigr).
\]
It is split when $\delta(R)$ is a square in $k$, and otherwise is the
spectrum of the quadratic field $k(\sqrt{\delta(R)})$.
\end{example}

\begin{remark}\label{rem:galois-odd}
Suppose that $\psi$ is geometrically Galois of odd degree $n$, for
instance a Kummer cover $w^\ell=f$ of odd prime degree $\ell$. Above a
point $b\in E(\bar k)$ the Galois group permutes the fibre transitively,
so all points of the fibre have the same ramification index $e$, which
divides $n$ and is therefore odd. Hence
$m_b=(n/e)(e-1)$ is even for every $b$, and $B_\psi=\varnothing$.
Thus Theorem~\ref{thm:main}(1) applies to every geometrically Galois cover
of odd degree $n$ with $g(C)>n^2+1$.
\end{remark}

\subsection{A positive density regular set of primes}

Let $\mathcal P_{\mathrm{all}}$ denote the set of rational primes, and
let $\mathcal P\subseteq\mathcal P_{\mathrm{all}}$. We use the same
notion of regularity as Khawaja and Siksek
\cite[Section~2]{KS}, due to Serre \cite{Ser75}. The set $\mathcal P$
is called \emph{regular of density} $\alpha>0$ if, for
$\operatorname{Re}(s)>1$,
\[
 \sum_{p\in\mathcal P}p^{-s}
 =
 \alpha\log\frac1{s-1}+\theta(s),
\]
where $\theta$ extends holomorphically to
$\operatorname{Re}(s)\geq1$. 
\newpage
It is called \emph{Frobenian of density}
$\alpha>0$ if there exist a finite Galois extension $L/\Q$ and a
subset $\mathcal C\subseteq G:=\Gal(L/\Q)$ such that
\begin{itemize}
\item $\mathcal C$ is a union of conjugacy classes in $G$;
\item $\alpha=\#\mathcal C/\#G$;
\item for every sufficiently large prime $p$, one has $p\in\mathcal P$
if and only if a Frobenius element $\sigma_p$ belongs to
$\mathcal C$.
\end{itemize}
Every Frobenian set of positive density is regular of the same density
\cite[Proposition~1.5]{Ser75}.

\begin{lemma}\label{lem:avoid}
Let $N\geq1$, let $U=\Spec\Z[1/N]$ and 
\(
 \mathcal B\to U
\)
be a nonempty finite \'etale scheme, and write $B=\mathcal B_\Q$ for
its generic fibre and $\mathcal B_p$ for the special fibre at a prime $p\nmid N$. 

Let $L/\Q$ be the finite Galois extension cut out by
the action of $G_\Q$ on $B(\Qbar)$, and write
$G_B=\Gal(L/\Q)$. Suppose that the action of $G_B$ on $B(\Qbar)$ contains a
derangement. Then
\[
 \alpha_B
 :=
 \frac{\#\{\sigma\in G_B:B(\Qbar)^\sigma=\varnothing\}}{\#G_B}
\]
lies in $(0,1)$, and there is a Frobenian, hence regular, set
$\mathcal P_B$ of primes $p\nmid N$, of density $\alpha_B$, such
that
\(
 \mathcal B_p(\FF_p)=\varnothing
\)
for every $p\in\mathcal P_B$.
\end{lemma}
\begin{proof}
 Since
$\mathcal B\to U$ is finite \'etale, the action of $G_\Q$ on
$B(\Qbar)$ factors through the
\'etale fundamental group of $U$. In particular, $L/\Q$ is unramified
at every prime $p\nmid N$.
Let
\[
 \mathcal D_B
 :=
 \{\sigma\in G_B:B(\Qbar)^\sigma=\varnothing\}.
\]
By hypothesis, $\mathcal D_B$ is nonempty. It is stable under
conjugation, since
\[
 B(\Qbar)^{\tau\sigma\tau^{-1}}
 =
 \tau\bigl(B(\Qbar)^\sigma\bigr)
\]
for all $\sigma,\tau\in G_B$. Moreover, since $\mathcal B\to U$ is
finite \'etale and $U$ is connected, the rank of $\mathcal B$ over
$U$ is constant, so the generic fibre $B$ is nonempty. Hence the
identity does not belong to $\mathcal D_B$, and
$0<\alpha_B<1$.

Let $\mathcal P_B$ be the set of primes $p\nmid N$ whose arithmetic
Frobenius conjugacy class in $G_B$ is contained in $\mathcal D_B$. By
Chebotar{\"e}v's Density Theorem~\cite[p.~34]{Chebotarev}, it has Frobenian (and hence regular) density
\(\frac{\#\mathcal D_B}{\#G_B}=\alpha_B.\)

For $p\nmid N$, choose a place of $\Qbar$ above $p$. Since
$\mathcal B\to U$ is finite \'etale, reduction at this place identifies
$B(\Qbar)$ with $\mathcal B_p(\overline{\FF}_p)$. Under this
identification, arithmetic Frobenius on the special fibre corresponds,
up to conjugacy, to a Frobenius element, $\Frob_p\in G_B$.
For $p\in\mathcal P_B$, the corresponding element $\Frob_p$ is a derangement. Since the
rational points of the special fibre are precisely its geometric points
fixed by arithmetic Frobenius,
\[
 \mathcal B_p(\FF_p)
 =
 \mathcal B_p(\overline{\FF}_p)^{\Frob_p}
 =
 \varnothing
\]
for every $p\in\mathcal P_B$.
\end{proof}
\newpage
\section{Specialisation and discriminant parity}
\label{sec:parity}

Recall our setting from the previous section and consider $k=\Q$. 
Let $C$ and $E$ be smooth projective geometrically integral curves
over $\Q$, and let
$\psi:C\to E$ be a finite morphism of degree $n$.
The discriminant ideal $\mathfrak d_\psi$  gives rise to the divisor 
\[
 D_\psi=\sum_jm_j[b_j],
\]
where the $b_j$ are closed points of $E$. 

In the proof of Theorem~\ref{thm:main} a crucial point is relating the discriminant of our cover, $D_\psi$, to discriminant ideals of algebras obtained by specialising $\psi$ at rational points of
$E$. This section formalises this idea, and in doing so we first need to choose a good integral model for the curve $C$.
\begin{lemma}[A good integral model]\label{lem:spread}
There is a finite set of primes $\Sigma$ such that, for
$U=\Spec\Z[1/\Sigma]$, there are smooth proper models $\mathcal C/U$
and $\mathcal E/U$ with geometrically integral fibres and generic fibres
$C$ and $E$, respectively, and a finite locally free morphism
\[
 \Psi:\mathcal C\to\mathcal E
\]
of degree $n$ extending $\psi$. Its discriminant ideal
defines an effective Cartier divisor $D_\Psi$ on $\mathcal E$ with
generic fibre $D_\psi$. Let $\mathcal Z_j$ be the reduced closure in
$\mathcal E$ of the closed point $b_j$. After enlarging $\Sigma$,
\[
 D_\Psi=\sum_jm_j\mathcal Z_j,
\]
so that $D_\Psi$ has no vertical components, each $\mathcal Z_j\to U$
is finite \'etale, and the $\mathcal Z_j$ are pairwise disjoint. In
particular,
\[
 \mathcal B_\psi
 :=
 \coprod_{\substack{j\\m_j\text{ odd}}}\mathcal Z_j
\]
is a closed subscheme of $\mathcal E$ that is finite \'etale over
$U$, with generic fibre $B_\psi$.
\end{lemma}
\begin{proof}

By the spreading out theorem \cite[Theorem~3.2.1]{Poo}, after replacing $\Spec \Z$ by a nonempty open subscheme $U=\Spec \Z[1/\Sigma],$ the curves $C$ and $E$ admit models $\mathcal C$ and $\mathcal E$ over $U$ by part~(i), which may be taken smooth, proper with geometrically integral fibres by part~(ii). By parts~(iii) and~(iv), respectively, $\psi$ extends to a morphism \[ \Psi:\mathcal C\longrightarrow\mathcal E \] which may be taken finite and flat. Hence $\Psi$ is finite locally free of degree $n$: its degree is locally constant on the integral scheme $\mathcal E$, and its generic fibre has degree $n$. Note that as $U $ is open in $\Spec \Z$, it follows that $\Sigma$ is a finite set of primes.

The local discriminant determinants differ on overlaps by squares of
units, so they define a locally principal ideal sheaf
$\mathfrak d_\Psi$. The trace pairing commutes with base change, so
its restriction to the generic fibre is
$\mathfrak d_\psi\neq0$. Since $\mathcal E$ is integral, the local
generators are non-zero divisors. Thus $\mathfrak d_\Psi$ is
invertible and defines an effective Cartier divisor $D_\Psi$, whose
generic fibre is $D_\psi$.
Thus, we may write \[ D_\Psi=\sum_j m_j\mathcal Z_j+V, \] where $V$ is supported on finitely many vertical components. After enlarging $\Sigma$ by the finitely many primes supporting $V$, we therefore have $D_\Psi=\sum_j m_j\mathcal Z_j$.

For each $j$, the generic fibre of $\mathcal Z_j\to U$ is $\Spec\Q(b_j)\to\Spec\Q$, which is finite \'etale. By \cite[Theorem~3.2.1(iv)]{Poo}, after shrinking $U$ further, each $\mathcal Z_j\to U$ is finite \'etale. Moreover, for $i\neq j$, the intersection $\mathcal Z_i\cap\mathcal Z_j$ has empty generic fibre. Since it is proper over $U$, its image in $U$ is closed and does not contain the generic point, hence is finite. Removing these images makes the $\mathcal Z_j$ pairwise disjoint. Since only finitely many components occur, one common shrinking gives all the assertions.
\end{proof}
\subsection{Specialisations and fibre discriminants}
Let $R\in E(\Q)$, and extend it by properness to a section
$s_R:U\to\mathcal E$. Form the pullback of
$\Psi:\mathcal C\to\mathcal E$ along $s_R$:
\[
\begin{array}{ccc}
\mathcal C_R:=\mathcal C\times_{\mathcal E}U & \longrightarrow & \mathcal C \\
\downarrow\scriptstyle{\pi_R} & & \downarrow\scriptstyle{\Psi} \\
U & \xrightarrow{\ s_R\ } & \mathcal E.
\end{array}
\]
Since finite locally free morphisms are preserved under base change,
$\pi_R:\mathcal C_R\to U$ is finite locally free of degree $n$. Let
\[
\mathcal A_R:=(\pi_R)_*\mathcal O_{\mathcal C_R},
\qquad
A_R:=\Gamma(U,\mathcal A_R).
\]
Since $\pi_R$ is finite and $U$ is affine,
$\mathcal C_R\simeq\Spec A_R$. Thus $\mathcal A_R$ is a locally
free $\mathcal O_U$-algebra of rank $n$. Equivalently, by base change
for finite morphisms,
\[
\mathcal A_R\simeq s_R^*(\Psi_*\mathcal O_{\mathcal C}).
\]
For $R\in E(\Q)$, let
\[
C_R:=C\times_{E, R}\Spec \Q=\psi^{-1}(R)
\]
denote the scheme-theoretic fibre of $\psi$ above $R$.
\begin{lemma}\label{lem:order}
With the notation above, suppose that there exists
$P\in C(\overline{\mathbb Q})$ such that $\psi(P)=R$ and
$[\mathbb Q(P):\mathbb Q]=n$.
Then $C_R\simeq\operatorname{Spec}\mathbb Q(P)$ as
$\mathbb Q$-schemes. Moreover, $A_R$ is a full
$\mathbb Z[1/\Sigma]$-order in $\mathbb Q(P)$.
\end{lemma}

\begin{proof}
Since $\psi$ is finite locally free of degree $n$, the
coordinate algebra $\Gamma(C_R,\mathcal O_{C_R})$
has dimension $n$ over $\mathbb Q$.
Moreover, evaluation at $P$ gives a surjective $\mathbb Q$-algebra
homomorphism
$\Gamma(C_R,\mathcal O_{C_R})\to\mathbb Q(P)$.
Indeed, its image is a finite-dimensional integral
$\mathbb Q$-algebra, hence a field, and is precisely
the field of definition of $P$.
Since $[\mathbb Q(P):\mathbb Q]=n$, this homomorphism
is an isomorphism. Thus
$C_R\simeq\operatorname{Spec}\mathbb Q(P)$.

The generic fibre of $\mathcal C_R\to U$ is $C_R$, so
$A_R\otimes_{\mathbb Z[1/\Sigma]}\mathbb Q
\simeq\mathbb Q(P)$.
Since $A_R$ is finite locally free over
$\mathbb Z[1/\Sigma]$, it is torsion free and therefore
embeds into $\mathbb Q(P)$ under this identification.
Every element of $A_R$ is integral over
$\mathbb Z[1/\Sigma]$, whose integral closure in
$\mathbb Q(P)$ is $\mathcal O_{\mathbb Q(P)}[1/\Sigma]$.
Consequently,
$A_R\subseteq\mathcal O_{\mathbb Q(P)}[1/\Sigma]$.
Together with
$A_R\otimes_{\mathbb Z[1/\Sigma]}\mathbb Q
=\mathbb Q(P)$,
this shows that $A_R$ is a full
$\mathbb Z[1/\Sigma]$-order in $\mathbb Q(P)$.
\end{proof}

We now show that in the setting of Lemma~\ref{lem:order}, taking the discriminant of a finite locally free algebra commutes with specialisation.

\begin{lemma}\label{lem:bc}
Let $\Psi:\mathcal C\to\mathcal E$ be finite locally free, and let
$\mathcal A:=\Psi_*\mathcal O_{\mathcal C},$
with discriminant ideal $\mathfrak d_\Psi\subset\mathcal O_{\mathcal E}$. For a section $s:U\to\mathcal E$, let
$\mathcal A_s:=s^*\mathcal A$. Then
\[
 \mathfrak d(\mathcal A_s/\mathcal O_U)
 =
 s^{-1}\mathfrak d_\Psi\cdot\mathcal O_U.
\]
If the generic fibre of $\mathcal A_s$ is \'etale, then the ideal on the
right is nonzero and
\[
 \operatorname{div}_U\bigl(\mathfrak d(\mathcal A_s/\mathcal O_U)\bigr)
 =
 s^*D_\Psi.
\]
In particular, for $p\notin\Sigma$,
\[
 v_p\bigl(\mathfrak d(\mathcal A_s/\mathcal O_U)\bigr)
 =
 i_{s(p)}\bigl(s(U),D_\Psi\bigr),
\]
where $i_{s(p)}$ denotes the local intersection multiplicity at
$s(p)\in\mathcal E$; see~\cite[Definition~9.1.1]{Liu}.
\end{lemma}

\begin{proof}
The assertion is local on $\mathcal E$ and $U$. Work on an affine
open over which $\mathcal A$ is free with basis $e_1,\ldots,e_n$.
After base change, the elements $e_1\otimes1,\ldots,e_n\otimes1$ form a
basis of $\mathcal A_s$. Trace commutes with base change, so the trace
matrix
\[
 \bigl(\operatorname{Tr}_{\mathcal A/\mathcal O_{\mathcal E}}(e_i e_j)\bigr)_{i,j}
\]
is carried to the trace matrix of $\mathcal A_s$. The same is therefore
true of its determinant, proving the identity of discriminant ideals.

If the generic fibre of $\mathcal A_s$ is \'etale, its discriminant is
nonzero. Hence $s^{-1}\mathfrak d_\Psi\cdot\mathcal O_U$ is nonzero,
so the pullback $s^*D_\Psi$ is defined. Since the ideal sheaf of
$s^*D_\Psi$ is $s^{-1}\mathfrak d_\Psi\cdot\mathcal O_U$, the first
assertion gives
\[
\operatorname{div}_U\bigl(\mathfrak d(\mathcal A_s/\mathcal O_U)\bigr)
=s^*D_\Psi.
\]

For $p\notin\Sigma$, let $f$ be a local equation for $D_\Psi$ at
$s(p)$. Let $s^\#: s^{-1}\mathcal{O}_{\mathcal E}\to \mathcal O_U$ be the morphism induced by $s$.
Then $s^\#(f)$ is a local equation for $s^*D_\Psi$ at
$p$, and hence
$
\operatorname{ord}_p(s^*D_\Psi)
=v_p\bigl(s^\#(f)\bigr).
$
Let $x=s(p)$. Since $D_\Psi$ is an effective Cartier divisor with
local equation $f$ at $x$, we have
\[
\mathcal O_{\mathcal E}(-D_\Psi)_x=(f)
\subseteq \mathcal O_{\mathcal E,x}.
\]
Moreover,
$\mathcal O_{\mathcal E}(-s(U))_x$ is the ideal of the section
$s(U)$ at $x$. Hence, by the definition of local intersection
multiplicity \cite[Definition~9.1.1]{Liu},
\[
i_x(s(U),D_\Psi)
=
\operatorname{length}_{\mathcal O_{\mathcal E,x}}
\frac{\mathcal O_{\mathcal E,x}}
{\mathcal O_{\mathcal E}(-s(U))_x+(f)}.
\]
Since $s:U\to s(U)$ is an isomorphism, we have
$
{\mathcal O_{\mathcal E,x}}/
{\mathcal O_{\mathcal E}(-s(U))_x}
\simeq \mathcal O_{U,p},
$
and under this identification the image of $f$ is $s^\#(f)$.
Therefore
\[
i_{s(p)}(s(U),D_\Psi)
=
\operatorname{length}_{\mathcal O_{U,p}}
\bigl(\mathcal O_{U,p}/(s^\#f)\bigr)
=
v_p(s^\#f).
\]
Thus
$
v_p\bigl(\mathfrak d(\mathcal A_s/\mathcal O_U)\bigr)
=i_{s(p)}\bigl(s(U),D_\Psi\bigr).
$
\end{proof}

\subsection{Reduction to rational fibres}

The preceding lemmas control the discriminant of a rational fibre. We now
show that, outside a finite exceptional set, primitivity forces every
relevant new point to arise from such a fibre. The following result is a generalisation of Vojta's Theorem \cite[Corollary 0.3]{Vojta}.

\begin{theorem}[Song--Tucker {\cite[Proposition~2.3]{ST}}]\label{thm:ST}
Let $C$ and $C'$ be curves of genera $g$ and $g_0$ over a number
field $k$, and fix
an algebraic closure $\bar k$ of $k$. Let $\nu\geq1$, and let $f:C\to C'$ be a dominant
$k$-morphism. If
\[
 g-1>(\nu+g_0-1)\deg f,
\]
then the set
\[
 \{P\in C(\bar k):[k(P):k]=\nu,\ k(f(P))=k(P)\}
\]
is finite.
\end{theorem}

Combining Theorem~\ref{thm:ST} with
Lemmas~\ref{lem:avoid} and~\ref{lem:spread}--\ref{lem:bc} gives the
arithmetic core of the proofs of Theorems~\ref{thm:main}
and~\ref{thm:field density}. Let $\mathcal P_{\mathrm{all}}$ be the set of all primes.

\begin{theorem}[Discriminant parity]\label{thm:parity}
Let $\psi:C\to E$ satisfy the hypotheses of Theorem~\ref{thm:main}.
There exist a finite set $\Sigma$ of rational primes and a finite set
$\mathcal F_{\mathrm{exc}}$ of $\Q$-isomorphism classes of degree $n$
fields with the following properties.
\begin{enumerate}
\item If $B_\psi=\varnothing$, then every primitive degree $n$ field
$L/\Q$ whose isomorphism class does not belong to
$\mathcal F_{\mathrm{exc}}$, and for which
$C(L)_{\news}\neq\varnothing$, satisfies
\[
 v_p(\Delta_L)\equiv0\pmod2
 \qquad\text{for every }p\notin\Sigma.
\]
\item If $B_\psi\neq\varnothing$, let $L_\psi/\Q$ be the finite
Galois extension cut out by the action of $G_\Q$ on
$B_\psi(\Qbar)$. Write
$G_\psi=\Gal(L_\psi/\Q)$ and
\[
 \alpha_\psi
 :=
 \frac{\#\{\sigma\in G_\psi:B_\psi(\Qbar)^\sigma=\varnothing\}}
      {\#G_\psi}.
\]
Then $0<\alpha_\psi<1$, and there is a regular set
$\mathcal S_0\subseteq\mathcal P_{\mathrm{all}}\setminus\Sigma$ of
density $\alpha_\psi$ such that every primitive degree $n$ field
$L/\Q$ whose isomorphism class does not belong to
$\mathcal F_{\mathrm{exc}}$, and for which
$C(L)_{\news}\neq\varnothing$, satisfies
\[
 v_p(\Delta_L)\equiv0\pmod2
 \qquad\text{for every }p\in\mathcal S_0.
\]
\end{enumerate}
The same conclusions hold, with the same proof, for a finite morphism $\psi:C\to Y$ of degree $n$ to a smooth projective geometrically integral curve $Y/\Q$ of genus $g_0$, provided that the hypothesis on $B_\psi$ is kept and the genus hypothesis is replaced by $g(C)-1>n(n+g_0-1)$.
\end{theorem}

\begin{proof}
Consider the set
\[
 \Xi_n(\psi)
 =
 \left\{
 P\in C(\Qbar):
 [\Q(P):\Q]=n,\ \Q(\psi(P))=\Q(P)
 \right\}.
\]
We start by applying Theorem~\ref{thm:ST} with $g_0=1$, $\nu=n$, and
$\deg\psi=n$. Its numerical hypothesis is exactly $g(C)>n^2+1$, so
$\Xi_n(\psi)$ is finite. Let $\mathcal F_{\mathrm{exc}}$ be the
finite set of $\Q$-isomorphism classes of the fields $\Q(P)$, with
$P\in\Xi_n(\psi)$.

Choose the integral model given by Lemma~\ref{lem:spread}, and retain
the resulting set $\Sigma$. If $B_\psi\neq\varnothing$, apply
Lemma~\ref{lem:avoid}, with $N=\prod_{p\in\Sigma}p$, to the finite
\'etale morphism
\[
 \mathcal B_\psi\to U=\Spec\Z[1/\Sigma].
\]
This gives a regular set
$\mathcal S_0\subseteq\mathcal P_{\mathrm{all}}\setminus\Sigma$ of
density $\alpha_\psi$ such that
\[
 \mathcal B_{\psi,p}(\FF_p)=\varnothing
 \qquad(p\in\mathcal S_0).
\]
For the rest of the proof, let
\[
 \mathcal S
 :=
 \begin{cases}
 \mathcal P_{\mathrm{all}}\setminus\Sigma,
   & B_\psi=\varnothing,\\
 \mathcal S_0,
   & B_\psi\neq\varnothing.
 \end{cases}
\]

Let $L/\Q$ be a primitive field of degree $n$, whose
$\Q$-isomorphism class does not belong to
$\mathcal F_{\mathrm{exc}}$, and choose
$P\in C(L)_{\news}$. Choose a $\Q$-embedding
$L\hookrightarrow\Qbar$, and use it to regard $P$ as a point of
$C(\Qbar)$. Let $R=\psi(P)$. Since
$\Q(R)\subseteq\Q(P)=L$ and $L/\Q$ is primitive, either
$\Q(R)=\Q$ or $\Q(R)=L$. The second possibility would imply
$P\in\Xi_n(\psi)$, contradicting the choice of $L$. Hence
$R\in E(\Q)$.

By properness, $R$ extends to a section $s_R:U\to\mathcal E$. Moreover, 
Lemma~\ref{lem:order} shows that the generic fibre above $R$ is
$\Spec L$ and that
$
 A_R
 =
 \Gamma\bigl(U,s_R^*(\Psi_*\mathcal O_{\mathcal C})\bigr)
$
is a full $\Z[1/\Sigma]$-order in $L$; in particular its generic fibre
$\Spec L$ is \'etale over $\Q$, as required in Lemma~\ref{lem:bc}.

Let
$
 D_\Psi=\sum_i m_i\mathcal Z_i
$
be as in Lemma~\ref{lem:spread}, and denote by
$$
 \mathfrak d(A_R)
 :=
 \mathfrak d\bigl(A_R/\Z[1/\Sigma]\bigr).$$
Then, for $p\in\mathcal S$, Lemma~\ref{lem:bc} gives
\[
 v_p\bigl(\mathfrak d(A_R)\bigr)
 =
 \sum_i m_i\,i_{s_R(p)}\bigl(s_R(U),\mathcal Z_i\bigr).
\]
Suppose that $m_i$ is odd. This cannot occur when
$B_\psi=\varnothing$. If $B_\psi\neq\varnothing$ and the
corresponding intersection multiplicity were positive, then
\[
 s_R(p)
 \in
 \mathcal Z_{i,p}(\FF_p)
 \subseteq
 \mathcal B_{\psi,p}(\FF_p),
\]
contrary to the definition of $\mathcal S_0$. Thus every term on the
right contributes an even integer, and for all $p\in\mathcal S$,
$
 v_p\bigl(\mathfrak d(A_R)\bigr)\equiv0\pmod2.
$

Finally, since $A_R$ is a full $\Z[1/\Sigma]$-order in $L$, the 
local index--discriminant formula (e.g. \cite[p.~226, (7-3)]{Stevenhagen}) gives, for every $p\notin\Sigma$,
\[
 v_p\bigl(\mathfrak d(A_R)\bigr)
 =
 v_p(\Delta_L)
 +
 2\operatorname{length}_{\Z_p}
 \left(
 \frac{\mathcal O_L\otimes_\Z\Z_p}
      {A_R\otimes_{\Z[1/\Sigma]}\Z_p}
 \right).
\]

Therefore for all $p\in\mathcal S$, we have that
$
 v_p(\Delta_L)\equiv0\pmod2,
$
which proves both assertions.
\end{proof}

\subsection{Proof of Theorem~\ref{thm:main}}

We first count the possible absolute discriminants.
Suppose that $B_\psi=\varnothing$, and let $\Sigma$ be the finite set
of primes given by Theorem~\ref{thm:parity}. Let
\[
 q_\Sigma:=\prod_{p\in\Sigma}p.
\]
Apart from the finitely many discriminants arising from
$\mathcal F_{\mathrm{exc}}$, every
$m\in\mathcal D_{C,n}(X)$ has even valuation at every
$p\notin\Sigma$. Hence 
$
 m=dr^2
$
with $d\mid q_\Sigma$ and $r\in\Z_{>0}$. It follows that
\[
 \#\mathcal D_{C,n}(X)
 \leq
 \#\mathcal F_{\mathrm{exc}}
 +\sum_{d\mid q_\Sigma}
   \left\lfloor\sqrt{\frac Xd}\right\rfloor
 \leq
 \#\mathcal F_{\mathrm{exc}}
 +X^{1/2}\sum_{d\mid q_\Sigma}\frac{1}{\sqrt d}
 \ll_\psi X^{1/2}.
\]
This proves assertion~(1) of Theorem~\ref{thm:main}.

Now suppose that $B_\psi\neq\varnothing$, and let $\mathcal S_0$ be
the regular set of primes of density $\alpha_\psi$ furnished by
Theorem~\ref{thm:parity}. Let
\[
 N_{\mathcal S_0}(X)
 :=
 \#\left\{ m \in \Z_{\geq 0}: \,
 m\leq X, \,
 v_p(m)\equiv0\pmod2
 \text{ for every }p\in\mathcal S_0
 \right\}.
\]
By \cite[Theorem~10]{KS}, there is a constant
$c_{\mathcal S_0}>0$ such that
\[
 N_{\mathcal S_0}(X)
 \sim
 c_{\mathcal S_0}\frac{X}{(\log X)^{\alpha_\psi}}.
\]
By Theorem~\ref{thm:parity}, every element of
$\mathcal D_{C,n}(X)$, apart from the finitely many absolute
discriminants arising from $\mathcal F_{\mathrm{exc}}$, is counted by
$N_{\mathcal S_0}(X)$. Hence
\[
 \#\mathcal D_{C,n}(X)
 \ll_\psi
 \frac{X}{(\log X)^{\alpha_\psi}}.
\]
This proves assertion~(2) of Theorem~\ref{thm:main}.
\begin{remark}\label{rem:general-target}
The proof uses only the fact that $E$ is a smooth projective geometrically
integral curve and its genus. More generally, let
$\psi:C\to Y$ be a finite morphism of degree $n$ to a curve of genus
$g_0$, and assume that either $B_\psi=\varnothing$, or that the Galois action on $B_\psi(\Qbar)$ contains a derangement. The conclusions of Theorems~\ref{thm:main}
and~\ref{thm:field density} hold, with the same definition of the odd locus, under the
numerical hypothesis
\[
 g(C)-1>n(n+g_0-1).
\]
For $Y=\PP^1$, this becomes $g(C)>n^2-n+1$, and
Proposition~\ref{prop:KS-dictionary} recovers the projective line
condition of \cite[Theorem~9]{KS}. For $g_0=1$, it is precisely the
hypothesis of Theorem~\ref{thm:main}.
The same bound occurs in Castelnuovo--Severi: if $g(C)>n^2+1$, then
any two degree $n$ maps from $C$ to elliptic curves factor through a
common map of degree greater than $1$ \cite[Theorem~3.11.3]{Sti}; for
prime $n$ they therefore differ by an isomorphism of their targets.
This observation is not used in the proofs.
If $g(Y)\geq 2$, then $Y(\mathbb Q)$ is finite by Faltings'
Theorem. Assuming that $\Xi_n(\psi)$ is finite, all but finitely many primitive degree $n$ points of $C$ lie above $Y(\Q)$. Since $\psi$ is finite, there are therefore only
finitely many such points, making our density question vacuous in this
case.
\end{remark}

\begin{remark}\label{rem:Xi-finite}
The numerical genus hypothesis is used only to prove, by
Theorem~\ref{thm:ST}, that
\[
\Xi_n(\psi)
:=
\left\{
P\in C(\Qbar):
[\Q(P):\Q]=n,\ 
\Q(\psi(P))=\Q(P)
\right\}
\]
is finite.  Thus, with all the other hypotheses unchanged, the conclusions of Theorems~\ref{thm:parity},
\ref{thm:main}, and~\ref{thm:field density} remain valid whenever the finiteness of
$\Xi_n(\psi)$ is known by another method.
\end{remark}

\section{Quartic and quintic fields}
\label{sec:qqext}

In this section, we count densities of fields with prescribed parity conditions on their discriminants, for degrees $4$ and $5$. We then use these results, together with Theorem~\ref{thm:parity}, to prove Theorem~\ref{thm:field density} and Corollary~\ref{cor:P1}.
Throughout this section, number fields are counted up to
$\Q$-isomorphism.
\subsection{Tame local masses}
Recall that every degree $n$ \'etale $\mathbb Q_p$-algebra has a decomposition
$
A=\prod_{i=1}^r K_i
$
into finite extensions $K_i/\mathbb Q_p$, unique up to permutation. Writing $\mathcal O_A=\prod_i\mathcal O_{K_i}$, we define $\operatorname{Disc}(A/\mathbb Q_p)$ to be the discriminant ideal of $\mathcal O_A/\mathbb Z_p$, equivalently the discriminant of its trace pairing. Thus
$
v_p\!\left(\operatorname{Disc}(A/\mathbb Q_p)\right)
 =\sum_i v_p\!\left(\operatorname{Disc}(K_i/\mathbb Q_p)\right).
$

Isomorphism classes of degree $n$ \'etale $\Q_p$-algebras correspond
to $S_n$-conjugacy classes of continuous homomorphisms $
\rho:G_{\Q_p}\to S_n$;
see \cite[Lemma~3.1]{Kedlaya}. Suppose that $p>n$. Since wild inertia
is pro-$p$ and $p\nmid n!$, its image under $\rho$ is trivial.
Thus $\rho$ factors through the tame quotient. In that quotient,
choose a topological generator $t$ of inertia and a lift
$\operatorname{Frob}_p$ of arithmetic Frobenius. Denote by 
$\tau:=\rho(t)$ and $\phi:=\rho(\operatorname{Frob}_p)$. Then we have
\begin{equation}\label{eq:conjugated}
\phi\tau\phi^{-1}=\tau^p.
\end{equation}
The splitting type of $A$ is the multiset $\{(e_i,f_i)\}_i$, where $e_i$ and $f_i$ are the ramification index and residue degree of $K_i/\mathbb Q_p$; in particular, $\sum_i e_if_i=n$. We write this type as $(e_1,f_1)+\cdots +(e_r,f_r)$. If all the factors are tamely ramified, then
\begin{equation}\label{eq:discA}
v_p\!\left(\operatorname{Disc}(A/\mathbb Q_p)\right)
=\sum_i f_i(e_i-1)=\sum_i f_ie_i-\sum_i f_i=n-\#\{\text{cycles of }\tau\}.
\end{equation}
The following calculation is a tame specialisation of Bhargava's mass
formula \cite[Theorem~1.1]{BhaLocalMass}, using Kedlaya's description
in terms of local Galois representations \cite{Kedlaya}. For a prime
$p$ and an integer $n\geq1$, let $\mathcal E_n(p)$ be a set of
representatives for the isomorphism classes of degree $n$ \'etale
$\Q_p$-algebras, and let
\[
 M_n(p)
 :=
 \sum_{A\in\mathcal E_n(p)}
 \frac{p^{-v_p(\Disc(A/\Q_p))}}
      {\#\operatorname{Aut}_{\Q_p}(A)}.
\]

\begin{lemma}
\label{lem:qqext-local-masses}
For $p>3$, one has
\[
 M_4(p)=1+p^{-1}+2p^{-2}+p^{-3},
\]
and the contribution from the algebras having even discriminant
exponent is
\[
 M_4^{\mathrm{even}}(p)=1+2p^{-2}.
\]
For $p>5$, one has
\[
 M_5(p)=1+p^{-1}+2p^{-2}+2p^{-3}+p^{-4},
\]
and the contribution from the algebras having even discriminant
exponent is
\[
 M_5^{\mathrm{even}}(p)=1+2p^{-2}+p^{-4}.
\]
\end{lemma}

 \begin{proof}
In each case $p>n$, so every algebra under consideration is tame.
Let $e(\tau):=n-\#\{\text{cycles of }\tau\}$, be the index of $\tau$ (where we count fixed points
as cycles). By \eqref{eq:discA}, this is its discriminant exponent. We now apply
Kedlaya's tame mass formula
\cite[Proposition~5.3 and Corollary~5.4]{Kedlaya} to $S_n$. It
assigns the mass
$p^{-e(\tau)}$ to each conjugacy class for which $\tau^p$ is
conjugate to $\tau$. Every class has this property, since all cycle
lengths are at most $n<p$.

A direct way to see this is as follows. There are $n!/\#C_{S_n}(\tau)$ choices for an
inertia element in its conjugacy class and, for each choice,
$\#C_{S_n}(\tau)$ Frobenius elements satisfying
$\phi\tau\phi^{-1}=\tau^p$. Dividing the number of labelled pairs
by $n!$ is equivalent, by the orbit-stabiliser theorem, to weighting each
isomorphism class by the reciprocal of its automorphism group.
Thus the contribution is exactly $p^{-e(\tau)}$.

The cycle types and their indices are given in the following tables.
\[
\begin{array}{c@{\qquad}ccccc}
\toprule
n=4&1^4&2\,1^2&3\,1&2^2&4\\
\midrule
e(\tau)&0&1&2&2&3\\
\bottomrule
\end{array}
\]
\[
\begin{array}{c@{\qquad}ccccccc}
\toprule
n=5&1^5&2\,1^3&2^2\,1&3\,1^2&4\,1&3\,2&5\\
\midrule
e(\tau)&0&1&2&2&3&3&4\\
\bottomrule
\end{array}
\]
Summing up over all conjugacy classes gives the values for $M_n(p)$, while summing over the even $e(\tau)$ gives $M_n^{\mathrm{even}}(p)$.
\end{proof}

\begin{lemma}\label{lem:q4q5}
Fix $n\in\{4,5\}$. For primes $p>n$, write $
 q_n(p):=\frac{M_n^{\mathrm{even}}(p)}{M_n(p)}$. Let $\mathcal S$ be a regular set of primes of positive density,
all greater than $n$, and let
\[
 T_1\subseteq T_2\subseteq\cdots\subseteq\mathcal S,
 \qquad
 \mathcal S=\bigcup_{m\geq1}T_m,
\]
be an increasing sequence of finite subsets. Then
\(
 \lim_{m\to\infty}\prod_{p\in T_m}q_n(p)=0.
\)
\end{lemma}

\begin{proof}
By Lemma~\ref{lem:qqext-local-masses}, we have $0<q_n(p)<1$ with
\[
 q_n(p)=1-\frac1p+O(p^{-2}), \text{ and } 
 \log q_n(p)=-\frac1p+O(p^{-2}).
\]
Moreover, regularity of positive density implies that 
$\sum_{p\in\mathcal S}p^{-1}=\infty$. Since
$\sum_p p^{-2}<\infty$, we have
\[
 \log\prod_{p\in T_m}q_n(p)
 =-\sum_{p\in T_m}\frac1p
  +O\left(\sum_{p\in T_m}\frac1{p^2}\right)
 \to-\infty,
\]
as $m \to \infty$.
Exponentiating the above equation proves the result.
\end{proof}

\subsection{Fields with parity-restricted discriminant}

In this section we count fields of degrees $5$ and $4$, ordered by discriminant, with local specifications.

We start with the quintic counting theorem. The proof follows the finite local condition argument used by
Khawaja--Siksek~\cite[Corollary~18]{KS} for cubic fields (relying on the cubic counting theorem in Bhargava--Taniguchi--Thorne~\cite{BTT}). In our case, we use the quintic local factors supplied by Lemma~\ref{lem:qqext-local-masses} and the counting
theorem of McGown--Thorne--Tucker~\cite{MTT23}, which is an extension of 
\cite[Theorems~2.4 and~5.1]{EPW17}.
\begin{theorem}
\label{thm:qqext-quintic-parity}
Let $\mathcal S$ be a regular set of rational primes of positive
density.  Then
\[A_{5,\mathcal S}(X):=
 \#\left\{L/\Q:
 \begin{array}{l}
 [L:\Q]=5,\quad |\Delta_L|\leq X,\\
 v_p(\Delta_L)\equiv0\pmod2
       \text{ for every }p\in\mathcal S
 \end{array}
 \right\}
 =o(X).
\]
In particular, these fields have density $0$ among all quintic fields
ordered by absolute discriminant.
\end{theorem}

\begin{proof}
Without loss of generality, we may assume that $\mathcal S$ contains no prime $p\leq5$.
Let $T\subseteq\mathcal S$ be finite. For each prime $p\in T$ let 
\[
\mathcal{E}_p
=
\left\{
A/\mathbb{Q}_p \text{ \'etale of degree } 5 :
v_p\bigl(\operatorname{Disc}(A/\mathbb{Q}_p)\bigr)
\equiv 0 \pmod{2}
\right\}.
\]
Because $p>5$, every such algebra is tamely ramified, and thus whether the exponent is even depends only on the splitting type
$\bigl((e_j,f_j)\bigr)_j$ by \eqref{eq:discA}. Let $\mathcal{T}_p^{\mathrm{even}}$ denote the set of splitting types giving $v_p(\Disc (A/\Q_p))$ even. Hence
\[
\mathcal E_p
=
\bigsqcup_{\lambda\in\mathcal T_p^{\mathrm{even}}}
\Sigma_p(\lambda),
\]
where $\Sigma_p(\lambda)$ is the set of local algebras having the single
splitting type $\lambda$.

Choose a tuple
\[
 \boldsymbol\lambda=(\lambda_p)_{p\in T}
 \in\prod_{p\in T}\mathcal T_p^{\mathrm{even}},
\]
and let
\[
 N(X;\boldsymbol\lambda)
 :=\#\left\{L/\Q:
 \begin{array}{l}
 [L:\Q]=5,\quad |\Delta_L|\leq X,\\
 L\otimes_{\Q}\Q_p\in\Sigma_p(\lambda_p)
       \text{ for every }p\in T
 \end{array}
 \right\}.
\]
We apply \cite[Theorem~6]{MTT23} with
$e=\prod_{p\in T}p$ and
$\Sigma=(\Sigma_p(\lambda_p))_{p\mid e}$, and sum over the three
signatures $i\in\{0,1,2\}$. This gives
\[
 N(X;\boldsymbol\lambda)
 =c_5X\prod_{p\in T}c_p(\lambda_p)+o_T(X),
\]
where $c_5:=\sum_{i=0}^2 C^{(i)}>0$, in the notation of
\textit{loc. cit.}, and
\[
 c_p(\lambda_p)
 =\frac{1}{M_5(p)}
  \sum_{A\in\Sigma_p(\lambda_p)}
  \frac{p^{-v_p(\Disc(A/\Q_p))}}
       {\#\operatorname{Aut}_{\Q_p}(A)}.
\]
The cited theorem uses $|\Delta_L|<Y$; taking
$Y=\lfloor X\rfloor+1$ gives the convention used here without
changing the asymptotic. For fixed $T$, there are only finitely
many allowed tuples, so the sum of their error terms is still
$o_T(X)$. Let
\[
 N_{5,T}(X)
 :=\#\left\{L/\Q:
 \begin{array}{l}
 [L:\Q]=5,\quad |\Delta_L|\leq X,\\
 v_p(\Delta_L)\equiv0\pmod2
       \text{ for every }p\in T
 \end{array}
 \right\}.
\]
The disjoint decomposition of $\mathcal E_p$ by splitting type and
Lemma~\ref{lem:qqext-local-masses} give
\[
 \sum_{\lambda_p\in\mathcal T_p^{\mathrm{even}}}c_p(\lambda_p)
 =\frac{M_5^{\mathrm{even}}(p)}{M_5(p)}
 =q_5(p).
\]
Summing over all
$\boldsymbol\lambda\in\prod_{p\in T}\mathcal T_p^{\mathrm{even}}$
therefore gives
\begin{equation}
\label{eq:qqext-quintic-fixed}
 N_{5,T}(X)
 =c_5\left(\prod_{p\in T}q_5(p)\right)X+o_T(X).
\end{equation}
For every fixed finite $T\subseteq\mathcal S$, we have
$
 A_{5,\mathcal S}(X)\leq N_{5,T}(X),
$
and hence by \eqref{eq:qqext-quintic-fixed}
\[
 \limsup_{X\to\infty}\frac{A_{5,\mathcal S}(X)}X
 \leq c_5\prod_{p\in T}q_5(p).
\]
By Lemma~\ref{lem:q4q5}, 
the product on the right tends to zero as $T$ increases through the
finite subsets of $\mathcal S$.  
Thus $A_{5,\mathcal S}(X)=o(X)$.

Finally, \cite[Theorem~6]{MTT23}, with no local conditions and summed
over the three signatures, gives $c_5X+o(X)$ quintic fields of
absolute discriminant at most $X$.  This proves the density
assertion.
\end{proof}
The analogous result for quartic fields relies on a theorem of Bhargava--Shankar--Wang \cite{BSWglobal} counting $S_4$-quartic fields; the negligible family of
$A_4$-quartic fields is treated separately.
\begin{theorem}
\label{thm:qqext-quartic-parity}
Let $\mathcal S$ be a regular set of rational primes of positive
density.  
Then
\[A_{4,\mathcal S}(X):=
 \#\left\{L/\Q:
 \begin{array}{l}
 [L:\Q]=4,\quad |\Delta_L|\leq X, L \text{ primitive}\\
 v_p(\Delta_L)\equiv0\pmod2
       \text{ for every }p\in\mathcal S
 \end{array}
 \right\}
 =o(X).
\]
In particular, these fields have density $0$ both among
all quartic fields and among the primitive quartic fields, in each case
ordered by absolute discriminant.
\end{theorem}

\begin{proof}
As above, we may delete the primes $p\leq 4$ from $\mathcal S$.
Let $T\subseteq\mathcal S$ be finite.  At every $p\in T$, prescribe
that the degree $4$ \'etale $\Q_p$-algebra have even discriminant
exponent, and impose no local condition elsewhere.  This is an
acceptable collection of local specifications in the sense of
\cite[Theorem~2]{BSWglobal}. 

This theorem counts $S_4$-quartic
fields, rather than arbitrary quartic orders or all quartic fields.
The fields are weighted by
$\#\operatorname{Aut}_{\Q}(L)^{-1}$, which is $1$ for an
$S_4$-quartic (see \cite[p.~2]{BSWglobal}). Denote by
\[
N_{G,T}(X)
:=
\#\left\{
L/\mathbb{Q}:
\begin{array}{l}
[L:\mathbb{Q}]=4,\quad |\Delta_L|\leq X,
\operatorname{Gal}(\widetilde{L}/\mathbb{Q})\simeq G,\\
v_p(\Delta_L)\equiv0\pmod{2}\text{ for every }p\in T
\end{array}
\right\}.
\]
Theorem~2 in \cite{BSWglobal} and Lemma~\ref{lem:qqext-local-masses} then give
\begin{equation}
\label{eq:qqext-quartic-fixed}
 N_{S_4,T}(X)
 =\left(c_{S_4}\prod_{p\in T}q_4(p)+o_T(1)\right)X,
\end{equation}
where $c_{S_4}>0$ is the leading constant for unrestricted
$S_4$-quartic fields.
The unrestricted and restricted finite local factors are
$(1-p^{-1})M_4(p)$ and
$(1-p^{-1})M_4^{\mathrm{even}}(p)$, respectively. Their ratio is
\[
 \frac{(1-p^{-1})M_4^{\mathrm{even}}(p)}
      {(1-p^{-1})M_4(p)}
 =q_4(p).
\]

A quartic field is primitive (i.e. it has no nontrivial
proper subfield) if and only if the Galois group of its normal closure
is $S_4$ or $A_4$.  Indeed, by classical group theory, the transitive subgroups of $S_4$ are
$S_4,A_4,D_4,V_4$, and $C_4$, and only $S_4$ and $A_4$ act
primitively in their natural degree $4$ actions. By
\cite[Theorem~1.4]{BSTTTZ}, we get $N_{A_4}(X):=N_{A_4,\emptyset}(X)=o(X)$. 

For every fixed finite
$T\subseteq\mathcal S$, we have
$
 A_{4,\mathcal S}(X)
 \leq N_{S_4,T}(X)+N_{A_4}(X).
$
It follows from \eqref{eq:qqext-quartic-fixed} that
\[
 \limsup_{X\to\infty}\frac{A_{4,\mathcal S}(X)}X
 \leq c_{S_4}\prod_{p\in T}q_4(p).
\]
As before, Lemma~\ref{lem:q4q5} gives
$A_{4,\mathcal S}(X)=o(X)$.

With no local conditions, \cite[Theorem~2]{BSWglobal} gives
$N_{S_4,\varnothing}(X)\sim c_{S_4}X$, where $c_{S_4}>0$.
Both the family of all quartic fields and the family of primitive
quartic fields contain these fields. Since
$A_{4,\mathcal S}(X)=o(X)$, both asserted density statements follow.
\end{proof}

\subsection{Consequences for covers of elliptic curves}

We start with a result which concerns the fields of definition of quintic new points.

\begin{corollary}
\label{cor:qqext-quintic-new-points}
Suppose that $\psi:C\to E$ has degree $5$ and satisfies the
hypotheses of Theorem~\ref{thm:main}.  Then
\[
 \#\{L/\Q:[L:\Q]=5,\ C(L)_{\news}\neq\varnothing,
              \ |\Delta_L|\leq X\}
 =o(X).
\]
In particular, these fields have density $0$ among all quintic
fields ordered by absolute discriminant.
\end{corollary}

\begin{proof}
Every degree $5$ field is primitive.  If $B_\psi=\varnothing$,
Theorem~\ref{thm:parity} supplies the cofinite regular set
$\mathcal P_{\mathrm{all}}\setminus\Sigma$, of density $1$.  If
$B_\psi\neq\varnothing$, it supplies a regular set of density
$\alpha_\psi>0$.  Apart from the finitely many fields in
$\mathcal F_{\mathrm{exc}}$, every relevant field has even
discriminant valuation at every prime in the corresponding regular
set.  The result follows from
Theorem~\ref{thm:qqext-quintic-parity}.
\end{proof}

\begin{corollary}
\label{cor:qqext-quartic-new-points}
Suppose that $\psi:C\to E$ has degree $4$ and satisfies the
hypotheses of Theorem~\ref{thm:main}.  Then
\[
 \#\left\{L/\Q:
 \begin{array}{l}
 [L:\Q]=4,\quad L/\Q\text{ primitive},\\
 C(L)_{\news}\neq\varnothing,\quad |\Delta_L|\leq X
 \end{array}
 \right\}
 =o(X).
\]
In particular, these fields have density $0$ both among all quartic
fields and among the primitive quartic fields.
\end{corollary}

\begin{proof}
Theorem~\ref{thm:parity} supplies either a cofinite regular set of
primes of density $1$, or a regular set of density
$\alpha_\psi>0$, on which every relevant primitive quartic field,
apart from finitely many exceptions, has even discriminant valuation.
The result follows from Theorem~\ref{thm:qqext-quartic-parity}.
\end{proof}

\subsection{Proofs of the field density statements}
\label{subsec:field density-proofs}
\begin{proof}[Proof of Theorem~\ref{thm:field density}]
We start with $n=2$.
Every quadratic field is primitive, and the
signed discriminant determines the field. Thus at most two
$\Q$-isomorphism classes of quadratic fields have any given absolute
discriminant (one for each sign). 
Thus
\[
\#\bigl\{L/\mathbb Q:[L:\mathbb Q]=2,\,
C(L)_{\mathrm{new}}\ne\varnothing,\,
|\Delta_L|\le X\bigr\}
\le
2\#\mathcal D_{C,2}(X)
=
o(X).
\]
Since the number of quadratic fields of absolute discriminant at most
$X$ is $\asymp X$, the quadratic fields $L$ for which
$C(L)_{\news}\neq\varnothing$ have density $0$.

Suppose now that $n=3$. Every cubic field is primitive. Define
\[
 \mathcal S
 :=
 \begin{cases}
 \mathcal P_{\mathrm{all}}\setminus\Sigma,
   & B_\psi=\varnothing,\\
 \mathcal S_0,
   & B_\psi\neq\varnothing,
 \end{cases}
\]
where $\Sigma$ and $\mathcal S_0$ are supplied by
Theorem~\ref{thm:parity}. The set $\mathcal S$ is regular of positive
density: its density is $1$ in the first case and $\alpha_\psi$ in
the second. Apart from finitely many exceptions, every cubic field $L$
for which $C(L)_{\news}\neq\varnothing$ satisfies
\[
 v_p(\Delta_L)\equiv0\pmod2
 \qquad(p\in\mathcal S).
\]
By \cite[Corollary~18]{KS}, cubic fields satisfying these parity
conditions have density $0$ among all cubic fields ordered by absolute
discriminant. Adding finitely many exceptional fields does not affect the
density. 
The cases $n=5$ and $n=4$ follow from Corollaries~\ref{cor:qqext-quintic-new-points} and~\ref{cor:qqext-quartic-new-points} respectively.
\end{proof}
\begin{proof}[Proof of Corollary~\ref{cor:P1}]
By Proposition~\ref{prop:KS-dictionary}, the hypotheses say that the action
of $G_\Q$ on $B_\psi(\Qbar)$ contains a derangement (if $g$ is
constant then $B_\psi=\varnothing$). Theorem~\ref{thm:parity} applies
to $\psi$ with $g_0=0$, since $g(C)-1>n(n-1)$, and supplies a
regular set of primes of positive density at which, apart from finitely
many exceptions, every relevant field has even discriminant valuation. The
conclusions follow from Theorems~\ref{thm:qqext-quintic-parity}
and~\ref{thm:qqext-quartic-parity} exactly as in
Corollaries~\ref{cor:qqext-quintic-new-points}
and~\ref{cor:qqext-quartic-new-points}.
\end{proof}

\section{Quadratic points on bielliptic modular curves}
\label{sec:modular-quadratic}

We now prove Corollary~\ref{cor:X0}.  
Recall first that a smooth projective curve $C/\Q$ is called \emph{bielliptic} if it
admits a degree $2$ morphism \[C\longrightarrow E\]
over $\Q$, where $E/\Q$ is an elliptic curve.  By \cite{Bars1999} (see also \cite[Section~4]{Bars}), the modular curves $X_0(N)$ of genus at
least $2$ admitting such a morphism to an elliptic curve of positive
Mordell--Weil rank are precisely the following:
\[
\begin{array}{c|cccccccccc}
g(X_0(N)) & 2 & 3 & 4  & 5 & 6 & 7  & 8 & 11\\ \hline
N & 37 & 43 & 53, 61 & 65 & 79 & 83, 89 & 101 & 131 
\end{array}
\]
For every value in the table one may take
\[
X_0^+(N):=X_0(N)/\langle w_N\rangle
\]
as the elliptic quotient, and $X_0^+(N)(\Q)$ has rank $1$;
see \cite[p. 214]{Galbraith}.  Here $w_N$ is the Atkin--Lehner involution attached to $N$. Under the moduli interpretation, $w_N$ sends a pair $(E,\ker \alpha)$ with 
$
\alpha\colon E \to E'
$
a cyclic
$N$-isogeny, to $(E',\ker \widehat{\alpha})$, where $\widehat \alpha \colon E' \to E$ is 
the dual isogeny. Note that at $N=37$ the curve is, uniquely among the ten levels, also hyperelliptic, which obstructs our method; see Remark~\ref{rmk:37-43}.

For a $D<0$, let $\mathcal O_D$ denote
the imaginary quadratic order of discriminant $D$, and $h(D)$ its class number. Moreover, let $$S_D=\{\,j(E) : \operatorname{End}_{\overline{\mathbb{Q}}}(E)\simeq\mathcal{O}_D\,\}.$$

\begin{lemma}\label{lem:fixed-points}
Let $N>4$ and let $\operatorname{Fix}(w_N)$ denote the fixed locus of $w_N$
on $X_0(N)$. 
Then $(E,\ker \alpha)\mapsto j(E)$ is a $G_{\mathbb{Q}}$-equivariant bijection
\[
  \operatorname{Fix}(w_N)(\overline{\mathbb{Q}})\;\to\;
  \begin{cases}
    S_{-4N}, & N\not\equiv 3\pmod 4,\\[2pt]
    S_{-N}\sqcup S_{-4N}, & N\equiv 3\pmod 4 .
  \end{cases}
\]
\end{lemma}

\begin{proof}
By \cite[Proposition~3 and p.~454]{Ogg}, $w_N$ fixes no cusp for
$N>4$, and its noncuspidal fixed points are precisely the pairs
$
(E,\ker\alpha),
$
where $\alpha\in\End_{\overline{\mathbb Q}}(E)$ has cyclic kernel of
order $N$ and satisfies
$
\alpha^2=[-N].
$

Let $\mathcal O=\End_{\overline{\mathbb Q}}(E)$ and let
$
m=[\mathcal O:\mathbb Z[\alpha]].
$
Since $\ker\alpha$ is cyclic, $\alpha$ is primitive in $\mathcal O$:
otherwise, if $\alpha=r\gamma$ with $r>1$, then
$E[r]\subseteq\ker\alpha$, contradicting cyclicity. Choose an integral
basis $1,\omega$ of $\mathcal O$ and write
\[
\alpha=a+m\omega.
\]
The primitivity of $\alpha$ gives $\gcd(a,m)=1$, while
$\Tr(\alpha)=0$ gives
\[
2a+m\Tr(\omega)=0.
\]
Hence $m\mid 2$. Since
$
\Delta(\mathbb Z[\alpha])=-4N
       =m^2\Delta(\mathcal O),
$
we obtain
\[
\Delta(\mathcal O)=
\begin{cases}
-4N,&m=1,\\
-N,&m=2.
\end{cases}
\]
The second case can occur only when $N\equiv3\pmod4$: indeed, if
$m=2$, then $\gcd(a,2)=1$, so the trace of $\omega$ is odd and
$\Delta(\mathcal O)\equiv1\pmod4$.

Conversely, if $\mathcal O=\mathcal O_{-4N}$, or if
$N\equiv3\pmod4$ and $\mathcal O=\mathcal O_{-N}$, then
$\sqrt{-N}\in\mathcal O$ has norm $N$ and cyclic kernel, and the
corresponding point of $X_0(N)$ is fixed by $w_N$. Moreover, the
elements of $K=\mathbb Q(\sqrt{-N})$ having trace $0$ and norm $N$
are exactly $\pm\sqrt{-N}$, which have the same kernel. Thus $\ker \alpha$ is
uniquely determined by $E$, and $(E,\ker \alpha)\mapsto j(E)$ gives the stated
bijection.

Finally, this map is defined over $\mathbb Q$, since $j(E^\sigma)=j(E)^\sigma$  for all $\sigma\in G_{\mathbb Q}$. Therefore, the bijection is $G_{\mathbb Q}$-equivariant.
\end{proof}

\begin{lemma}\label{lem:AL-derangement}
Let $N>4$ and let
$\psi_N\colon X_0(N)\to X_0^{+}(N):=X_0(N)/\langle w_N\rangle$ be the
quotient by the Atkin--Lehner involution.  Let $B_N:=B_{\psi_N}$ be the odd
part of the discriminant divisor of $\psi_N$.  Then:
\begin{enumerate}
  \item if $N\not\equiv 3\pmod 4$, the action of $G_{\mathbb{Q}}$ on
        $B_N(\overline{\mathbb{Q}})$ contains a derangement if and only if
        $h(-4N)>1$;
  \item if $N\equiv 3\pmod 4$, the action of $G_{\mathbb{Q}}$ on
        $B_N(\overline{\mathbb{Q}})$ contains a derangement if and only if
        $h(-N)>1$.
\end{enumerate}

\end{lemma}
 
\begin{proof}
Since $\psi_N$ is the quotient by an involution in characteristic zero,
it is ramified precisely at the fixed points of $w_N$, each with ramification index $2$, so by Lemma~\ref{lem:disc-ramification} each branch point occurs with coefficient $1$ in the discriminant divisor.  Consequently $\psi_N$ induces a
$G_{\mathbb{Q}}$-equivariant bijection
$\operatorname{Fix}(w_N)(\overline{\mathbb{Q}})\to B_N(\overline{\mathbb{Q}})$,
and Lemma~\ref{lem:fixed-points} therefore identifies $B_N(\overline{\Q})$ with
$S_{-4N}$, respectively $S_{-N}\sqcup S_{-4N}$, as $G_{\mathbb{Q}}$-sets.
 
 By the theory of complex multiplication (e.g. \cite[Corollary 10.20]{Cox}), we get that
$\#S_D=h(D)$. Moreover, the field generated by the
elements of $S_D$ over $K=\mathbb{Q}(\sqrt{-N})$ is the ring class field $H_D$, and
$\operatorname{Gal}(H_D/K)\simeq\operatorname{Pic}(\mathcal{O}_D)$ acts
simply transitively on $S_D$; see for example \cite[Corollary 11.37]{Cox}.  In particular the
action of $G_K$ on $S_D$ factors through $\operatorname{Gal}(H_D/K)$, and
{every} non-trivial element of $\operatorname{Gal}(H_D/K)$ acts on
$S_D$ without fixed points.
 
Suppose first that $N\not\equiv 3\pmod4$, so
$B_N(\overline{\mathbb{Q}})\simeq S_{-4N}$.  If $h(-4N)=1$, this is a
$G_{\mathbb{Q}}$-stable singleton, so no derangement exists.  If
$h(-4N)>1$, choose $\tau\ne 1$ in $\operatorname{Gal}(H_{-4N}/K)$ and lift it
to $\tilde\tau\in G_K$; then $\tilde\tau$ is a
derangement.
 
Suppose now that $N\equiv 3\pmod 4$, so
$B_N(\overline{\mathbb{Q}})\simeq S_{-N}\sqcup S_{-4N}$.  If $h(-N)=1$ then
$S_{-N}$ is a $G_{\mathbb{Q}}$-stable singleton and no derangement exists.
Conversely, suppose that $h(-N)>1$.  The corresponding orders satisfy
$\mathcal O_{-4N}\subseteq\mathcal O_{-N}$ with index $2$, so $
  H_{-N}\subseteq H_{-4N}
$
(e.g. \cite[Exercise~9.19]{Cox}). Restriction gives a surjection
\[
\Gal(H_{-4N}/K)
\twoheadrightarrow
\Gal(H_{-N}/K).
\]
Choose
$
\tau\in\Gal(H_{-4N}/K)
$
whose restriction to $H_{-N}$ is nonidentity.  Then $\tau$ is
itself nonidentity, so it fixes no point of $S_{-4N}$.
Its restriction to $H_{-N}$ is also nonidentity, and therefore it
fixes no point of $S_{-N}$.
Choose a lift
$
\widetilde\tau\in G_K
$
once again to get the desired derangement.
\end{proof}

\begin{remark}\label{rem:class-number-one}
The negative discriminants of class number one are
$-3$, $-4$, $-7$, $-8$, $-11$, $-12$, $-16$, $-19$, $-27$, $-28$, $-43$, $-67$, $-163$. If $N>4$ and
$N\not\equiv3\pmod4$, then $-4N\leq-20$ is not in this list, so
$h(-4N)>1$. Consequently Lemma~\ref{lem:AL-derangement} shows that, for
$N>4$, the action of $G_\Q$ on $B_N(\Qbar)$ contains a derangement
unless $N\in\{7,11,19,27,43,67,163\}$. Among the ten levels in the table
above, only $N=43$ is excluded for this reason.
\end{remark}

\begin{proof}[Proof of Corollary~\ref{cor:X0}]
Let
$
N\in\{53,61,65,79,83,89,101,131\}.
$
The quotient
\[
\psi_N:X_0(N)\longrightarrow X_0^+(N)
\]
has degree $2$, and $X_0^+(N)(\Q)$ has rank $1$.  By
Lemma~\ref{lem:AL-derangement} and Remark~\ref{rem:class-number-one}, the
action of $G_\Q$ on $B_N(\Qbar)$ contains a derangement.

Suppose first that
$
N\in\{79,83,89,101,131\}.
$
The corresponding genera are $6,7,7,8,11$, respectively, and hence
$
g(X_0(N))>5.
$
The hypotheses of Theorem~\ref{thm:main} are therefore satisfied.

It remains to consider
$
N\in\{53,61,65\}
$
of genera $4,4,5$.  
Box proves that every quadratic point
$P\in X_0(N)(\overline{\mathbb Q})$ lies above a rational point of
$X_0^+(N)$ under the quotient map
$\psi_N:X_0(N)\to X_0^+(N)$; see
\cite[\S\S~4.2, 4.3 and 4.5]{Box}. Thus
$\mathbb Q(\psi_N(P))=\mathbb Q$, whereas
$[\mathbb Q(P):\mathbb Q]=2$. Consequently,
$P\notin\Xi_2(\psi_N)$, and hence
$
\Xi_2(\psi_N)=\varnothing.
$
Remark~\ref{rem:Xi-finite} therefore allows the proof of
Theorem~\ref{thm:main} to be applied without the numerical genus
hypothesis.

In every case, the parity argument supplies a regular set of primes of
density $\alpha_N>0$ at which the discriminant valuation of every
relevant quadratic field is even, apart from finitely many exceptions.
The quadratic field argument in the proof of
Theorem~\ref{thm:field density}, together with
Theorem~\ref{thm:main}, then gives
\[
\#\{
L/\Q\text{ quadratic}:
|\Delta_L|\leq X,\ 
X_0(N)(L)_{\news}\neq\varnothing
\}
\ll_N
\frac{X}{(\log X)^{\alpha_N}}.
\]
\end{proof}
\begin{remark}\label{rmk:37-43}
As noted above, $X_0(37)$ is both hyperelliptic and bielliptic, with
$X_0^+(37)$ of rank $1$. Write $\iota$ for the hyperelliptic
involution. For a quadratic point $P$, let $\sigma$ denote the nontrivial automorphism of $\Q(P)/\Q$.  Such a point $P$ in a fibre of the hyperelliptic map
satisfies $P^{\sigma}=\iota(P)$, so $\psi_{37}(P)\in X_0^{+}(37)(\Q)$
if and only if $P^\sigma=w_{37}(P)$, that is, if and only if $P$ is
fixed by $\iota w_{37}$. The quotient of $X_0(37)$ by $\iota w_{37}$
is an elliptic curve (of rank $0$; see \cite[p.~340]{Box}), so by
Riemann--Hurwitz $\iota w_{37}$ has exactly two fixed points. Hence all
but finitely many quadratic points of the hyperelliptic pencil lie in
$\Xi_2(\psi_{37})$, which is therefore infinite, and neither
Theorem~\ref{thm:main} nor the variant of Remark~\ref{rem:Xi-finite}
applies. This mirrors the exclusion of the same level from
\cite[Corollary~5]{KS}, where the obstruction is that $J_0(37)(\Q)$ is
infinite; its rank is exactly that of the bielliptic quotient
$X_0^{+}(37)$.

For $X_0(43)$ the derangement hypothesis fails: indeed $h(-43)=1$, so
by Lemmas~\ref{lem:fixed-points} and~\ref{lem:AL-derangement} the single branch point in
$S_{-43}$ is rational and fixed by every element of $G_\Q$.
\end{remark}

The quadratic points on the curves in Corollary~\ref{cor:X0} have
previously been described. Work of Box \cite[\S\S~4.2, 4.3 and 4.5]{Box} treated
the levels $53,61,65$, while Najman and Vukorepa
\cite[\S\S~9.2--9.4, 9.6 and~9.8]{NajmanVukorepa} treated
$79,83,89,101,$ $131$; in each case every quadratic point arises from a
rational point on the Atkin--Lehner quotient. These results give a
geometric description of the quadratic points, but do not address the
distribution of their fields of definition. Moreover, 
Michaud-Jacobs
\cite[Theorem~1]{MichaudJacobs} studies the related question field by
field, determining for each $N$ above and each integer $d$ with
$\lvert d\rvert<100$ whether
$
X_0(N)(\mathbb Q(\sqrt d))\ne X_0(N)(\mathbb Q).
$
Corollary~\ref{cor:X0}, by contrast, gives an asymptotic result: although
these curves have infinitely many quadratic points, the quadratic
fields supporting new points form a density $0$ subset of all
quadratic fields ordered by absolute discriminant.

\begin{remark}[Modular covers of degrees $3$ and $5$]
\label{rem:higher-degree modular}
The LMFDB \cite{LMFDB} contains modular covers of positive rank elliptic
curves to which Theorem~\ref{thm:main} applies, and the hypotheses can be
checked from the data recorded there. Two examples are the covers
\[
\begin{aligned}
\href{https://www.lmfdb.org/ModularCurve/Q/72.216.13.h.1/}{\texttt{72.216.13.h.1}}
&\xrightarrow{\;3\;}
\href{https://www.lmfdb.org/ModularCurve/Q/24.72.1.h.1/}{\texttt{24.72.1.h.1}},\\
\href{https://www.lmfdb.org/ModularCurve/Q/40.480.33.ck.1/}{\texttt{40.480.33.ck.1}}
&\xrightarrow{\;5\;}
\href{https://www.lmfdb.org/ModularCurve/Q/40.96.1.o.2/}{\texttt{40.96.1.o.2}},
\end{aligned}
\]
whose sources have genera $13$ and $33$ and whose targets are elliptic curves (genus $1$ with a rational cusp) of Mordell--Weil rank~$1$: the first target is the elliptic curve
$y^2=x^3+8$
(\href{https://www.lmfdb.org/EllipticCurve/Q/576/e/4}{\texttt{576.e4}}),
the second an elliptic curve of conductor $800$ (isogeny class
\href{https://www.lmfdb.org/EllipticCurve/Q/800/d/}{\texttt{800.d}}).
We first note that cusps pull back to cusps, because they are exactly the points above $j=\infty$. Since source and target have the same number of geometric cusps ($12$ in the first pair and $16$ in the second), each cover is totally ramified above every cusp of its
target. The resulting contributions give $12\cdot(3-1)=24=2\cdot13-2$ and
$16\cdot(5-1)=64=2\cdot33-2$. Thus, by the Riemann--Hurwitz theorem, there
is no other ramification: every nonzero coefficient of the discriminant divisor is $\deg\psi-1$, which is even, and $B_\psi=\varnothing$ in both cases.
Since $13>3^2+1$ and $33>5^2+1$, Theorems~\ref{thm:main}(1)
and~\ref{thm:field density}(1) apply to both covers. The gonality bounds
$\gamma\geq4$ and $\gamma\geq8$ recorded in the LMFDB show, moreover,
that neither source admits a map of the same degree to $\PP^1$, even
over $\Qbar$. A systematic search of the database for further such
covers, and explicit equations for them, would make an interesting
computational project, in the direction of the third project of
\cite[Section~8]{KS}.
\end{remark}

\bibliographystyle{plain}
\bibliography{references}

\end{document}